%% file: main.tex
\documentclass[sigconf]{acmart}
\usepackage[utf8]{inputenc}
\usepackage{verbatim} 
\usepackage{graphicx} 
\usepackage{mathtools} 
\usepackage{bbm,dsfont} 
\usepackage{mathrsfs} 
\numberwithin{equation}{section} 
\usepackage{empheq} 
\usepackage{xfrac}
\usepackage{dsfont}

\usepackage{subcaption}

\usepackage{enumerate}
\usepackage{hyperref}

\usepackage{color}
\definecolor{red}{rgb}{.7,0,0}
\definecolor{blue}{rgb}{0,0,1}

\def\mcG{\mathcal{G}}

\def\mcM{\mathcal{M}}

\def\mcP{\mathcal{P}}

\def\mcZ{\mathcal{Z}}

\def\bbone{\mathds{1}}
\def\bbI{\mathbb{I}}
\def\bbE{\mathbb{E}}

\def\bbB{\mathbb{B}}
\def\bbR{\mathbb{R}}
\def\bbZ{\mathbb{Z}}
\def\bbN{\mathbb{N}}

\def\bbP{\mathbb{P}}
\def\bbS{\mathbb{S}}
\usepackage{ifpdf}
\ifpdf
\hypersetup{
 pdftoolbar=true,
 pdfmenubar=true,
 pdfstartview={FitV},
 pdfcreator={overleaf}, 
 pdfproducer={overleaf},
 linkcolor=red     
 citecolor=green   
 urlcolor=cyan     
 filecolor=magenta 
}
\fi
\usepackage{bookmark}

\makeatletter
\def\th@plain{%
  \thm@notefont{}
  \slshape 
}
\def\th@definition{%
  \thm@notefont{}
  \normalfont 
}
\makeatother
\theoremstyle{acmplain}
\newtheorem{lem}{Lemma}[section]
\newtheorem{prop}[lem]{Proposition}
\newtheorem{theo}[lem]{Theorem}
\newtheorem{cor}[lem]{Corollary}

\theoremstyle{acmdefinition}
\newtheorem{defi}[lem]{Definition}
\theoremstyle{remark}

\newtheorem{exam}[lem]{Example}

\newtheorem{remark}[lem]{Remark}
\newcommand{\eproof}{\hfill\qed}

\def\bfd{\boldsymbol{d}}

\def\Pd{\mcP_{n,\bfd}}

\def\hm{^{\mathsf{h}}}

\def\Tg{\mathrm{T}}

\def\diff{\mathrm{D}}

\DeclareMathOperator{\vol}{\mathrm{vol}}
\DeclareMathOperator{\Med}{\mcM}
\DeclareMathOperator{\reach}{\varrho}

\def\D{\mathsf{D}}
\def\hinfty{\mathsf{h}\infty}

\DeclareMathOperator{\cond}{\texttt{C}}
\def\dist{\mathrm{dist}}
\def\fkf{\mathfrak{f}}
\def\fkg{\mathfrak{g}}
\def\fkh{\mathfrak{g}}
\def\fkx{\mathfrak{x}}

\def\fkv{\mathfrak{v}}

\usepackage[ruled,linesnumbered,boxed,boxruled,vlined]{algorithm2e}
\makeatletter
\let\original@algocf@latexcaption\algocf@latexcaption
\long\def\algocf@latexcaption#1[#2]{%
  \@ifundefined{NR@gettitle}{%
    \def\@currentlabelname{#2}%
  }{%
    \NR@gettitle{#2}%
  }%
  \original@algocf@latexcaption{#1}[{#2}]%
}
\makeatother

\renewcommand\footnotetextcopyrightpermission[1]{} 
\setcopyright{none}
\acmConference[]{}{}{}


\begin{document}

\settopmatter{printacmref=false}

\title{Some Lower Bounds on the Reach of an Algebraic Variety}
\author{Chris La Valle}
\orcid{0009-0007-4825-6653}
\affiliation{
\institution{The University of Texas at San Antonio}
\department{Department of Mathematics}
\streetaddress{One UTSA Circle}
\city{San Antonio}
\state{Texas}
\postcode{78249}
\country{USA}}
\email{christopher.lavalle@my.utsa.edu}

\author{Josu\'{e} Tonelli-Cueto}
\orcid{0000-0002-2904-1215}
\affiliation{
\institution{The Johns Hopkins University}
\department{Department of Applied Mathematics and Statistics}
\streetaddress{}
\city{Baltimore}
\state{Maryland}
\postcode{}
\country{USA}}
\email{josue.tonelli.cueto@bizkaia.eu}

\input{0abstract.tex}

\maketitle

\newpage
\input{1TEXT.tex}

\subsubsection*{Acknowledgements} J.T.-C. is grateful to Evgenia Lagoda for moral support and to Jazz G. Suchen for useful suggestions for this paper. C.L.V. and J.T.-C. were partially supported by start-up funds of Alperen Ergür and NSF CCF 2110075.

\bibliographystyle{ACM-Reference-Format}
\bibliography{biblio.bib}
\end{document}

%% file: 0abstract.tex
\begin{abstract}
Separation bounds are a fundamental measure of the complexity of solving a zero-dimensional system as it measures how difficult it is to separate its zeroes. In the positive dimensional case, the notion of reach takes its place. In this paper, we provide bounds on the reach of a smooth algebraic variety in terms of several invariants of interest: the condition number, Smale's $\gamma$ and the bit-size. We also provide probabilistic bounds for random algebraic varieties under some general assumptions.
\end{abstract}

%% file: 1TEXT.tex
\section{Introduction}

The reach~\cite{federer1959} is a positive dimensional generalization of the separation between zeros. This quantity is fundamental in the homology inference for submanifolds~\cite{nswA,nswB}, and, in particular, for algebraic and semialgebraic sets~\cite{CKS16,BCL19,BCTC1,BCTC2}. Recently, a large amount of interest has emerged in the study of the reach of algebraic varieties~\cite{breidingkaivniksturmfels2018,horobetweinstein2019,diroccoeklundweinstein2020,horobet2024,diroccoeklundgafvert2022,diroccoedwardseklundgafverthauenstein2023,eklund2023}, but, aside from a few exceptions~\cite{CKS16,BCL19}, the existing lower bounds for the reach are scarce.

In this paper, we aim to close this gap in the literature. Building off the work of~\cite{CKS16,BCL19}, \cite{EPR-probcond1,EPR-probcond2,vershyninbook}, \cite{TCT-cube1conf,TCT-cube1,ETCT-descartes} and~\cite{jeronimoperruccitsigaridas2013}, we improve some of the existing lower bounds for the reach of an algebraic variety, providing new probabilistic bounds. As the reach can be seen as a positive dimensional analogue of the separation of the zeros, this can be seen as a generalization and expansion of~\cite{emirismourraintsigaridas2010,emirismourraintsigaridas2020} to the positive dimensional case. 

\subsection*{Notations} $\Pd[q]$ will denote the set of real polynomial $q$-tuples, where the $i$th polynomial has degree $\leq d_i$. $\bfd:=(d_1,\ldots,d_q)$, $\D:=\max \bfd$ denotes the maximum degree and $\Delta=\mathrm{diag}(\bfd)$ the diagonal matrix with $d_i$ in its $i$th entry. For a polynomial tuple $f\in\Pd[q]$ and $x\in\bbR^n$, $\mcZ(f)$ denotes the real zero, $\diff_xf$ the tangent map of $f$ at $x$, and $\diff^k_xf$ the $k$th derivative tensor of $f$. $\partial^k f$ will denote the corresponding tensor of polynomials formed by the partial derivatives of $f$, and $\partial^kf[v_1,\ldots,v_k]$ the polynomial obtained when contracting the polynomial tensor $\partial^kf$ with the vectors $v_1,\ldots,v_k$. We will denote random variables (in Fraktur) as $\fkf,\fkg,\fkh\ldots$ For $x\in \bbR^m$, $\|x\|_r$ is the $r$-norm and $\|x\|_{\hinfty}:=\max\{1,\|x\|_\infty\}$. For $A:(\bbR^a)^\ell\rightarrow \bbR^b$ multilinear, $\|A\|_{r,s}:=\sup_{v_1,\ldots,v_\ell}\|A(v_1,\ldots,v_\ell)\|_s/\prod_i\|v_i\|_r$ is the $(r,s)$-spectral norm. For linear $M$, $M^\dagger$ is the pseduinverse.

\section{Main Results}

The \emph{medial axis} of a closed set $Z\subseteq \bbR^n$, $\Med(X)$, is the set of those points with more than one nearest point ot $X$, i.e.,
\begin{equation}
    \Med(Z):=\left\{p\in\bbR^n\mid \#\{\zeta\in Z\mid \dist_2(p,\zeta)=\dist_2(p,Z)\}>1\right\}.
\end{equation}
The distance from $Z$ to its medial axis is the \emph{reach} of $Z$, $\reach(Z)$, i.e.,
\begin{equation}
    \reach(Z):=\dist_2(Z,\Med(Z))=\min\{\dist_2(\zeta,p)\mid (\zeta,p)\in Z\times \Med(Z)\};
\end{equation}
and the distance from $\zeta\in Z$ to the medial axis of $Z$ is the \emph{local reach} of $Z$ at $\zeta$, $\reach(Z,\zeta)$, i.e.,
\begin{equation}
    \reach(Z,\zeta):=\dist_2(\zeta,\Med(Z))=\inf\{\dist_2(\zeta,p)\mid p\in\Med(Z)\}.
\end{equation}
The reach is in this way a higher dimensional analogue of separation. We provide lower bounds for the following related quantity:
\begin{equation}
    \reach_R(Z):=\min\{\reach(Z,\zeta)\mid \zeta\in Z,\,\|\zeta\|_\infty\leq R\}.
\end{equation}

We provide bounds in three cases:

\begin{theo}\label{theo:worstcaseboundreach}
Let $R\in\bbN$ and $f\in\Pd[q]$ be a tuple of integer polynomials with coefficients of bit-size at most $\tau$. Then $\reach_R(\mcZ(f))=0$ or $\log 1/\reach_R(\mcZ(f))$
is (upper) bounded by
\[
4n(2\D)^{1+q+2n}(5+\tau+\log R+6n\log \D)+2\log\D+\tau.
\]
\end{theo}

\begin{theo}\label{theo:probcaseboundreach}
Let $\fkf\in\Pd[q]$ be a random integer polynomial tuple whose coefficients are independent and uniformly distributed in $\bbZ\cap [-2^\tau,2^\tau]$. Then for $t \in \left[\max\{\log\D,\log n\},\log q(n+1)\binom{n+\D}{n}+\tau-1\right]$,
\[\bbP\left(\log\frac{1}{\reach_R(\mcZ(\fkf))}\geq t\right)\leq  20 R n^{\frac{n}{2}+1}\left(\sqrt{2}\binom{n+\D}{n}\right)^{q+n}\D^{q+2n}2^{-t}.\]
\end{theo}

\begin{theo}\label{theo:probcaseboundreach_contunif}
Let $\fkf\in\Pd[q]$ be a random polynomial tuple whose coefficients are independent and uniformly distributed in $[-1,1]$. Then for $t\geq 2(n+1)$,
\[
\bbP\left(\log\frac{1}{\reach_R(\mcZ(\fkf))}\geq t\right)\leq 8R(n+1)^{\frac{n}{2}+1}\D^{q+2n}\left(4\binom{n+\D}{n}\right)^{n+q}\, t^{\frac{n+q}{p}}2^{-t}.\]
\end{theo}

The first theorem provides a worst case bound on the reach of an algebraic variety given by integer polynomials. The second and third provide probabilistic bounds for the reach in the case of a random polynomial with, respectively, random integer coefficients and random continuous real coefficients. The probabilistic theorems are, respectively, particular cases of Theorems~\ref{theo:reach_randombit} and~\ref{theo:reachtailbound_cont}.

Moreover, we provide bounds of the reach in terms of the condition number (Theorem~\ref{theo:condandreach}), Kantorovich's theory (Theorem~\ref{theo:kantorovichandreach}) and Smale's $\alpha$-theory (Theorem~\ref{theo:gammaandreach}).

\section{Federer's Estimators of the Reach}

The following results can be found in some form in \cite{federer1959}. As of today, they are fundamental in the estimation of the reach of manifolds, since they allow us to compute the reach. Moreover, the most widely used estimator for the reach~\cite{aamarikimfredericmichelrinaldowasserman2019} uses them.

\begin{theo}\label{theo:localreachestimate}
Let $Z\subseteq \bbR^m$ be a closed set, $\zeta\in Z$ and $r,t>0$.
If  $Z\cap B(\zeta,r)$ is a closed submanifold of $B(\zeta,r)$ and for all $\tilde{z},z\in Z\cap B(\zeta,r)$,
$\dist_2(\tilde{z}-z,\Tg_zZ)\leq \|\tilde{z}-z\|_2^2/(2t),$
then
$\reach(Z,\zeta)\geq \min\{r,t\}.$
\end{theo}
\begin{cor}\label{cor:localreachestimate}
Let $Z\subseteq \bbR^m$ be a locally closed submanifold. Then
\[\reach(Z)=\inf_{z,\tilde{z}\in Z}\frac{\|\tilde{z}-z\|_2^2}{2\dist_2(\tilde{z}-z,\Tg_zZ)}.\]
\end{cor}
\begin{proof}[Proof of Theorem~\ref{theo:localreachestimate}]
This is a variation of \cite[Lemma~4.17]{federer1959}. To see this variation, we only have to take $x,a,b\in B(\zeta,r)$ as in the proof of \cite[Lemma~4.17]{federer1959}, i.e., $a,b\in Z$ and $x$ such that $\dist_2(x,Z)=\|a-x\|_2=\|b-x\|_2<\min\{r,t\}$. Then the argument in the proof gives that $a=b$. Hence $\dist_2(\Med(Z),Z\cap B(\zeta,r))\geq \min\{r,t\}$ and so $\reach(Z,\zeta)\geq \min\{r,t\}$.
\end{proof}

\section{1-Norm on Polynomials}

Given a $q$-tuple of polynomials $f=\left(\sum_{\alpha}f_{i,\alpha}X^\alpha\right)_{i=1}^q\in\Pd[q]$, we consider the \emph{1-norm} $\|~\|_1$ given by
\begin{equation}
    \|f\|_1:=\max_{i}\sum_{\alpha}|f_{i,\alpha}|.
\end{equation}
The most important theorem regarding the 1-norm for polynomials is the following theorem relating the norm of $f$ and that of $\partial f$. To avoid confusion, we use $\partial f[v]$ to denote the $q$-tuple of polynomials given by $\diff_Xf[v]$, reserving the symbol $\diff$ to denote the Jacobian of $f$ (and higher derivatives) at a particular point. In other words, we use the symbol $\partial$ to emphasize that we are considering the polynomial.

\begin{theo}\label{theo:onenormineq}
Let $f\in\Pd[q]$ and $v\in\bbR^n$. Then
\begin{equation}
\left\|\Delta^{-1}\partial f[v]\right\|_1\leq \|f\|_1\|v\|_\infty
\end{equation}
where $\Delta=\mathrm{diag}(\bfd)$.
\end{theo}

The following corollaries follow from applying iteratively the theorem above. Recall that for a $\ell$-multilinear map $A:\bbR^{n_1}\times \cdots \times\bbR^{n_{\ell}}\rightarrow \bbR^m$, we define its \emph{$(p,q)$-spectral norm}, $\|~\|_{p,q}$, by
\begin{equation}
    \|A\|_{p,q}:=\sup_{v_1,\ldots,v_{\ell}\neq 0}\frac{\|A[v_1,\ldots,v_{\ell}]\|_q}{\|v_1\|_p\cdots\|v_{\ell}\|_p}
\end{equation}
where $A[v_1,\ldots,v_{\ell}]$ is the evaluation of $A$ at $(v_1,\ldots,v_k)$. Following the convention before, we denote by $\partial^{\ell}f[v_1,\ldots,v_{\ell}]$ the $q$-tuple of polynomials $\diff_X^{\ell}f[v_1,\ldots,v_{\ell}]$. 

\begin{cor}\label{cor:onenormineqA}
Let $f\in\Pd[q]$, $\ell\in\bbN$ and $v_1,\ldots,v_{\ell}\in\bbR^n$. Then
\begin{equation}
    \left\|\frac{1}{\ell!}\partial^{\ell}f[v_1,\ldots,v_{\ell}]\right\|_1\leq \binom{\D}{\ell}\|f\|_1\|v_1\|_\infty\cdots\|v_{\ell}\|_\infty.
\end{equation}
Moreover,
\begin{equation}
\left\|\frac{1}{\ell!}\partial^{\ell}f[v_1,\ldots,v_{\ell}]\right\|_1\leq \left\|\binom{\Delta}{\ell}f\right\|_1\|v_1\|_\infty\cdots\|v_{\ell}\|_\infty
\end{equation}
where $\binom{\Delta}{\ell}:=\Delta (\Delta-\bbI)\cdots (\Delta-(\ell-1)\bbI)/\ell!$.
\end{cor}

\begin{cor}\label{cor:onenormineqB}
Let $f\in\Pd[q]$, $x\in \bbR^n$ and $\ell\in\bbN$. Then
\begin{equation}
    \left\|\frac{1}{\ell!}\diff^\ell_x f\right\|_{\infty,\infty}\leq \binom{\D}{\ell}\|f\|_1\max\{1,\|x\|^{\D-\ell}\}.
\end{equation}
Moreover,
\begin{equation}
\left\|\|x\|_{\hinfty}^{-\bfd+\ell\bbone}\frac{1}{\ell!}\diff_x^\ell f\right\|_{\infty,\infty}\leq \left\|\binom{\Delta}{\ell}f\right\|_1
\end{equation}
where $\binom{\Delta}{\ell}$ as in Corollary ~\ref{cor:onenormineqA} and $\|x\|_{\hinfty}^{\bfd-\ell\bbone}:=\mathrm{diag}(\max\{1,\|x\|\}^{d_i-\ell})$.
\end{cor}

\begin{cor}\label{cor:onenormineqC}
Let $f\in\Pd[q]$, $k\in\bbN$ and $v_1,\ldots,v_k\in\bbR^n$. Then the function
\begin{align*}
[-1,1]^n\ni x&\mapsto \frac{1}{\ell!}\diff_x^{\ell}f\left[v_1/\|v_1\|_\infty,\ldots,v_{\ell}/\|v_{\ell}\|_\infty\right]\in \bbR^q
\end{align*}
is $(\ell+1)\left\|\binom{\Delta}{\ell+1}f\right\|_1$-Lipschitz with respect the $\infty$-norm, where $\binom{\Delta}{k}:=\frac{1}{k!}\Delta(\Delta-\bbI)\cdots (\Delta-(k-1)\bbI)$ being $\bbI$ the identity matrix. 
\end{cor}

\begin{proof}[Proof of Theorem~\ref{theo:onenormineq}]
Without loss of generality $\|v\|_\infty=1$. For any $i$,
$\|\partial f_i (v)\|_1= \left\|\sum_{k=1}^n\frac{\partial f_i}{\partial X_k}\,v_k\right\|_1\leq \sum_{k=1}^{n}\left\|\frac{\partial f_i}{\partial X_k}\right\|_1.$
Now,
$\left\|\frac{\partial f_i}{\partial X_k}\right\|_1=\sum_\alpha \alpha_k|f_{i,\alpha}|,$
thus
\[\sum_{k=1}^{n}\left\|\frac{\partial f_i}{\partial X_k}\right\|_1=\sum_{k=1}^n\sum_\alpha \alpha_k|f_{i,\alpha}|\leq \sum_\alpha\left(\sum_{k=1}^n\alpha_k\right)|f_{i,\alpha}|\leq d_i \|f_i\|_1\]
since the degree of $f_i$ being at most $d_i$ implies $\sum_{k=1}^n\alpha_k\leq d_i$. Hence
\[\left\|\Delta^{-1}\partial f\,v\right\|_1=\max_i\left\|d_i^{-1}\partial f_i\,v\right\|_1\leq \max_i\|f_i\|_1=\|f\|_1,\]
as desired.
\end{proof}
\begin{proof}[Proof of Corollary~\ref{cor:onenormineqA}]
We only need to prove the second family of inequalities. Note that, by substituting $f$ by $\Delta f$, we can rewrite Theorem~\ref{theo:onenormineq} as $\|\partial f[v]\|_1\leq \|\Delta f\|_1$. Now,
\begin{equation}\label{eq:iteratedderivative}
\partial^{\ell}f[v_1,\ldots,v_{\ell}]=\partial\left(\partial^{\ell-1}f[v_1,\ldots,v_{\ell-1}]\right)[v_{\ell}].   
\end{equation}
Hence the claim follows by induction, since $\binom{\Delta}{\ell}=\frac{\Delta}{\ell}\binom{\Delta-\bbI}{\ell-1}$.
\end{proof}
\begin{proof}[Proof of Corollary~\ref{cor:onenormineqB}]
This follows immediately from Corollary~\ref{cor:onenormineqA} and an elementary bound.
\end{proof}
\begin{proof}[Proof of Corollary~\ref{cor:onenormineqC}]
By~\eqref{eq:iteratedderivative} and the mean value theorem, we only need to show that
\[\left\|\frac{1}{\ell!}\partial^{\ell+1}f\left[\frac{v_1}{\|v_1\|_\infty,}\ldots,\frac{v_{\ell}}{\|v_{\ell}\|_\infty},w\right]\right\|\leq \left\|\Delta\binom{\Delta-\bbI}{\ell}f\right\|_1\|w\|_\infty.\]
Now, this follows from Corollary~\ref{cor:onenormineqB}. 
\end{proof}

\section{Kantorovich, Smale and the Reach}

Kantorovich's theory and Smale's $\alpha$-theory are at the core of certifying fast convergence of Newton's method. Now, they also appear in bounds of the reach.

\subsection{Kantorovich's Theory and the Reach}

Kantorovich's theory allows us to decide the fast convergence of Newton's method for a $C^2$-function using $f(x)$, $\diff_xf$ and the values of $\diff_z^2f$ around $x$ (see \cite[Th\'{e}or\`{e}me 88]{dedieubook} among others). Recall that $A^\dagger:=A^*(A A^*)^{-1}$ is the \emph{pseudoinverse} of a surjective matrix $A$.
\begin{defi}
Let $f:\bbR^n\rightarrow \bbR^q$ be a $C^2$-function and $x\in \bbR^n$, then $\beta(f,x):=\|\diff_xf^\dagger f(x)\|_2$ and \emph{Kantorovich's regularity measure} is
\begin{equation}
K(f,x,r):=\max\left\{\|\diff_xf^\dagger\diff_z^2f\|_{2,2}\mid \|z-x\|_2\leq r\right\}.  
\end{equation}
\end{defi}

We give a lower bound of the local reach in terms of the Kantorovich's regularity measure.

\begin{theo}\label{theo:kantorovichandreach}
Let $f:\bbR^n\rightarrow \bbR^q$ be a $C^2$-function, $\zeta\in\mcZ(f):=\{x\in\bbR^n\mid f(x)=0\}$ and $r>0$. If $\,\diff_\zeta f$ is surjective and $K(f,\zeta,r)r<1$, then
\[\reach(\mcZ(f),\zeta)\geq r.\]
In particular, if $1/2\leq K(f,\zeta,r)r\leq 1$, then $\reach(\mcZ(f),\zeta)\geq 1/(2K(f,\zeta,r))$.
\end{theo}
\begin{lem}\label{lem:matrixdaggerbound}
Let $A,B\in\bbR^{q\times n}$. If $B$ is surjective and $\|B^\dagger (A-B)\|_2<1$, then $A$ is surjective and
\[\|A^\dagger B\|_{2,2}\leq 1/(1-\|B^\dagger (A-B)\|_{2,2}).\]
\end{lem}
\begin{proof}[Proof of Theorem~\ref{theo:kantorovichandreach}]
We will apply Theorem~\ref{theo:localreachestimate}. Note that
\[\dist_2(\tilde{z}-z,\Tg_{z}\mcZ(f))=\|\diff_{z}f^\dagger \diff_{z}f[\tilde{z}-z]\|_2,\]
since $\diff_{z}f^\dagger \diff_{z}f$ is the orthogonal projection onto $(\ker \diff_{z}f)^\perp=(\Tg_{z}X)^{\perp}$. Hence, assuming $\|\diff_{\zeta}f^\dagger(\diff_zf-\diff_{\zeta} f)\|_{2,2}<1$,
\begin{equation}\label{eq:Kbasebound}
    \dist_2(\tilde{z}-z,\Tg_{z}\mcZ(f))\leq \frac{\|\diff_{\zeta}^\dagger\diff_{z}f[\tilde{z}-z]\|_2}{1-\|\diff_{\zeta}f^\dagger(\diff_zf-\diff_{\zeta} f)\|_{2,2}}
\end{equation}
due to $\diff_\zeta f$ being surjective, $\diff_\zeta f\diff_{\zeta}f^\dagger=\bbI$ and Lemma~\ref{lem:matrixdaggerbound}. Now, by Taylor's theorem with residue to obtain for all $i$, there is some $x_i\in[\tilde{z},z]:=\{t\tilde{z}+(1-t)z\mid t\in[0,1]\}$ such that
\begin{equation}\label{eq:K1}
-\diff_{z}f_i[\tilde{z}-z]=\frac{1}{2} \diff_{x_i}^{2}f_i[\tilde{z}-z,\tilde{z}-z],
\end{equation}
and for all $i$, some $y_i\in [z,\zeta]:=\{tz+(1-t)\zeta\mid t\in[0,1]\}$,
\begin{equation}\label{eq:K2}
\diff_zf_i-\diff_{\zeta} f_i= \diff_{y_i}^2f_i[z-\zeta].
\end{equation}
Hence, $[\tilde{z},z],[z,\zeta]\subseteq B(\zeta,r)$,
\begin{equation}\label{eq:K1max}
\|\diff_{\zeta}f^\dagger \diff_{z}f[\tilde{z}-z]\|_2\leq \frac{1}{2}K(f,\zeta,r)\|\tilde{z}-z\|_2^2
\end{equation}
and
\begin{equation}\label{eq:K2max}
\|\diff_{\zeta}f^\dagger(\diff_zf-\diff_{\zeta} f)\|_{2,2}\leq K(f,\zeta,r)r.
\end{equation}
The claim now follows combining these bounds, \eqref{eq:Kbasebound} and Theorem~\ref{theo:localreachestimate}, the claim follows.
\end{proof}
\begin{proof}[Proof of Lemma~\ref{lem:matrixdaggerbound}]
Since $B$ is surjective, then $B^\dagger$ is injective and so the rank of $A$ and $B^\dagger A$ is the same. Now, let $\sigma_q$ be the $q$th singular value, then
\[\sigma_q(B^\dagger A)\geq \sigma_q(B^\dagger B)-\|B^\dagger (A-B)\|_{2,2}=1-\|B^\dagger (A-B)\|_{2,2}>0\]
by the $1$-Lipschitz property of $\sigma_q$, the fact that $B^\dagger B$ is an orthogonal projection onto a $q$-dimensional subspace and the assumption. Therefore $\mathrm{rank}\, B^\dagger A\geq q$ and so $\mathrm{rank}\, A\geq q$ and $A$ is surjective therefore.

For the inequality, note that, since $BB^\dagger =\bbI$,
\[A^\dagger B=(B-(A-B))^\dagger B=(B(\bbI-B^\dagger(A-B)))^\dagger B.\]
Here, $\bbI-B^\dagger(A-B)$ is invertible with inverse $\bbI+\sum_{k=1}^\infty (-1)^kB^\dagger(A-B)$ since $\|B^\dagger(A-B)\|_{2,2}<1$. Thus
\[A^\dagger B=(B(\bbI-B^\dagger(A-B)))^\dagger B(\bbI-B^\dagger(A-B))(\bbI-B^\dagger(A-B))^{-1}.\]
Hence, since $(B(\bbI-B^\dagger(A-B)))^\dagger B(\bbI-B^\dagger(A-B))$ is an orthogonal projection,
\[\|A^\dagger B\|_{2,2}\leq \|(\bbI-B^\dagger(A-B))^{-1}\|_{2,2}\leq 1/(1-\|B^\dagger(A-B)\|_{2,2}),\]
by the series for the inverse.
\end{proof}

\subsection{Smale's \texorpdfstring{$\alpha$}{alpha}-theory and the Reach}

Smale's $\alpha$-theory allows us to decide the fast convergence of Newton's method for an analytic function with information depending on the value of $f$ and its derivatives at the initial point $x$ (see \cite[Th\'{e}or\`{e}me 128]{dedieubook} among others).

\begin{defi}
Let $f\in\Pd[q]$ and $x\in\bbR^n$. Then:
\begin{enumerate}
    \item[($\alpha$)] \emph{Smale's $\alpha$}: $\alpha(f,x):=\beta(f,x)\gamma(f,x)$.
    \item[($\beta$)] \emph{Smale's $\beta$}: $\beta(f,x):=\|\diff_xf^\dagger f(x)\|_{2}$.
    \item[($\gamma$)] \emph{Smale's $\gamma$}: $\gamma(f,x):=\sup_{\ell\geq 2}\left\|\diff_xf^\dagger \frac{1}{\ell!}\diff_x^\ell f\right\|_{2,2}$ if $\diff_xf$ is surjective, and zero otheriwse.
\end{enumerate}
\end{defi}

Again, we can obtain a lower bound of the local reach in terms of this new measure of regularity: Smale's $\gamma$. Even though we state the theorem for polynomials, it holds for more general analylitic functions. This improves the constant in~\cite[Theorem 3.3]{BCL19}.

\begin{theo}\label{theo:gammaandreach}
Let $f\in\Pd[q]$ and $\zeta\in\mcZ(f)$. Then
\[\reach(\mcZ(f),\zeta)\geq 1/(5\gamma(f,\zeta)).\]
\end{theo}
\begin{proof}
We could repeat the proof of Theorem~\ref{theo:kantorovichandreach} using Theorem~\ref{theo:localreachestimate}. However, it is easier and equivalent, to deduce this theorem from Theorem~\ref{theo:kantorovichandreach}. By expanding $\diff_z^2f$ in its Taylor series and adding up the norms, we have
\begin{equation}
    K(f,\gamma,r)\leq 2\gamma(f,\zeta)/(1-\gamma(f,\zeta)r)^3.
\end{equation}
For $r=t/\gamma(f,\gamma)$, we have $K(f,\gamma,r)r<2t/(1-t)^3$. Thus $K(f,\gamma,r)r<1$ if $t<0.22908\ldots$. Hence, by Theorem~\ref{theo:kantorovichandreach}, for $t=0.22908$, 
\begin{equation}
    \reach(\mcZ(f),\zeta)\geq r=0.22908/\gamma(f,\zeta)\geq 1/(5\gamma(f,\zeta)),
\end{equation}
as desired.
\end{proof}

\section{Cubic Condition Number and Reach}

We introduce a condition number. Instead of following~\cite{TCT-cube1}, where we only worked in the unit cube, we will be working now in the whole $\bbR^n$. Recall that $A^\dagger:=A^*(A A^*)^{-1}$ is the \emph{pseudoinverse} of a surjective matrix $A$ and that, by convention, $\|A^{\dagger}\|_{\infty,2}=\infty$ if $A$ is not surjective. Below, $\|x\|_{\hinfty}:=\max\{1,\|x\|_\infty\}$, the $\infty$-norm of the homogeneization $\begin{pmatrix}1\\x\end{pmatrix}$ of $x$; $\lambda^a:=\mathrm{diag}(\lambda^{a_1},\ldots,\lambda^{a_n})$ for $\lambda>0$ and $a\in \bbR^n$, and $\bbone:=(1,\ldots,1)$. 

\begin{defi}\label{def:onenormcond}
Let $f\in \Pd[q]$ and $x\in\bbR^n$. Then the \emph{local 1-condition number} is
\[\cond(f,x):=\frac{\|f\|_1}{\max\{\|\Delta^{-1}\|x\|_{\hinfty}^{-\bfd}f(x)\|_\infty,\|\diff_xf^\dagger\|x\|^{\bfd-\bbone}_{\hinfty}\Delta^2\|_{\infty,2}^{-1}\}},\]
and, for $R\in [1,\infty]$, the \emph{(global) 1-condition number} is 
\[
\cond_R(f):=\sup_{\|z\|_\infty\leq R}\cond(f,z)\in [1,\infty].
\]
\end{defi}

\subsection{Properties of the Condition Number}

\begin{theo}\label{theo:condnumberproperty}
Let $f,g\in\Pd[q]$ and $x,y\in\bbR^n$. Then the following hold:
\begin{enumerate}[(a)]
\item \textbf{Bounds}: $1\leq \cond(f,x)\leq \cond_R(f)\leq \infty$.
\item \textbf{Regularity Inequality}: Either 
\[
\|\Delta^{-1}\|x\|_{\hinfty}^{-\bfd}f(x)\|_\infty/\|f\|_1>1/\cond(f,x)
\]
or
\[
\|\diff_xf^\dagger\|x\|^{\bfd-\bbone}_{\hinfty}\Delta^2\|_{\infty,2}^{-1}/\|f\|_{1} >1/\cond(f,x).
\]
\item \textbf{1st Lipschitz Property}: 
\[
\left|\left\|g\right\|_1/\cond(g,x)-\left\|f\right\|_1/\cond(f,x)\right|\leq \|\Delta^{-1}(g-f)\|_1
\]
and
\[
\left|\left\|g\right\|_1/\cond_R(g)-\left\|f\right\|_1/\cond_R(f)\right|\leq \|\Delta^{-1}(g-f)\|_1.
\]
\item \textbf{2nd Lipschitz Property}: 
\[
\left|1/\cond(f,y)-1/\cond(f,x)\right|
\leq 2\|y-x\|_\infty/\|x\|_\infty
\]
\item \textbf{Higher Derivative Estimate}: If $\cond(f,x)\frac{\|\Delta^{-1}\|x\|_{\hinfty}^{-\bfd}f(x)\|_\infty}{\|f\|_1}< 1$, then
\[\gamma(f,x)\leq \D\cond(f,x)/\|x\|_{\hinfty}.\]
\end{enumerate}
\end{theo}
\begin{prop}\label{prop:condhomogeneous}
Let $f\in\Pd[q]$. Denote by $f\hm\in\mcP_{1+n,\bfd}[q]$ the \emph{homogeneization} of $f$ given by
\[f\hm_i:=f_i(X/X_0)X_0^{d_i}.\]
Then, for all $x\in\bbR^{n}$,
\[\cond(f,x)=\frac{\|f\|_1}{\max\left\{\left\|\Delta^{-1}f\hm\left(\frac{x}{\|x\|_{\hinfty}}\right)\right\|_\infty,\left\|\left(\diff_{\frac{x}{\|x\|_{\hinfty}}}f\hm P_0\right)^{\dagger}\Delta^2\right\|_{\infty,2}^{-1}\right\}}\]
where $P_0=\bbI-e_0e_0^*$. In particular,
\[\partial [-1,1]^{n+1}\ni \begin{pmatrix}x_0\\x\end{pmatrix}\mapsto \frac{1}{\cond(f,x/x_0)}\]
is $1$-Lipschitz with respect the $\infty$-norm.
\end{prop}
\begin{lem}\label{lem:minvalue}
Let $A\in\bbR^{q\times n}$ be surjective. Then
\[\|A^{\dagger}\|_{\infty,2}^{-1}=\min_{\substack{v\perp \ker A\\v\neq 0}}\frac{\|Av\|_\infty}{\|v\|_2}=\max_{\substack{V\leq\bbR^n\\\dim V={n-q}}}\min_{\substack{v\in V\\v\neq 0}}\frac{\|Av\|_\infty}{\|v\|_2}.\]
\end{lem}

\begin{proof}[Proof of Theorem~\ref{theo:condnumberproperty}]
\textbf{(a)} This follows from (c).
\textbf{(b)} This is trivial.
\textbf{(c)} By the reverse triangle inequality applied twice and Lemma~\ref{lem:minvalue}, $\left|\|g\|/\cond(g,x)-\|f\|_1/\cond(f,x)\right|$ is bounded by the maximum of
$\|\Delta^{-1}\|x\|_{\hinfty}^{-\bfd}(g-f)(x)\|_\infty$
and
\[\min_{v\neq 0,\,v\perp \ker A}\|\Delta^{-2}\|x\|_{\hinfty}^{-(\bfd-\bbone)}\diff_x(g-f)[v]\|_2/(\|v\|_\infty)\}.\]
Now, by Corollary~\ref{cor:onenormineqB}, these both are bounded by $\|\Delta^{-1}f\|_1$ and we are done.
\textbf{(d)} This follows from Proposition~\ref{prop:condhomogeneous} and
\[\left\|\frac{1}{\|y\|_{\hinfty}}\begin{pmatrix}1\\y\end{pmatrix}-\frac{1}{\|x\|_{\hinfty}}\begin{pmatrix}1\\x\end{pmatrix}\right\|_\infty\leq 2\frac{\|y-x\|_\infty}{\|x\|_{\hinfty}}.\]
\textbf{(e)} We have that
\[\left\|\diff_xf^{\dagger}\frac{1}{\ell!}\diff_x^{\ell}f\right\|_{2,2}\leq \left\|\diff_xf^{\dagger}\|x\|^{\bfd-\bbone}_{\hinfty}\Delta^2\right\|_{\infty,2}\left\|\|x\|^{\bbone-\bfd}_{\hinfty}\Delta{-2}\frac{1}{\ell!}\diff_x^{\ell}f\right\|_{2,\infty}\]
and, by $\|A\|_{2,\infty}\leq \|A\|_{\infty,\infty}$ and Corollary~\ref{cor:onenormineqB}, the second factor is bounded by
\[\left\|\|x\|^{\bbone-\bfd}_{\hinfty}\Delta{-2}\frac{1}{\ell!}\diff_x^{\ell}f\right\|_{\infty,\infty}\leq \|x\|_{\hinfty}^{-(\ell-1)}\D^{\ell-2}\|f\|_1.\]
Therefore, by (a), (b) and the above,
\begin{multline*}
\gamma(f,x)
\leq \sup_{\ell\geq 2}\cond(f,x)^{\frac{1}{\ell-1}}\D^{\frac{\ell-2}{\ell-1}}/\|x\|_{\hinfty}\leq \cond(f,x)\D/\|x\|_{\hinfty},
\end{multline*}
which is the claim we wanted to prove.
\end{proof}
\begin{proof}[Proof of Proposition~\ref{prop:condhomogeneous}]
The equality is an easy consequence of the definition of homogeneization. For the second part, consider for $x\in [-1,1]^{1+n}$,
\begin{equation}\label{eq:condhom}
\cond\hm(f,x):=\frac{\|f\|_1}{\max\left\{\left\|\Delta^{-1}f\hm(x)\right\|_\infty,\left\|\left(\diff_{x}f\hm P_0\right)^{\dagger}\Delta^2\right\|_{\infty,2}^{-1}\right\}}.    
\end{equation}
We only have to show that
\[[-1,1]^{1+n}\ni x\mapsto 1/\cond\hm(f,x)\]
is $1$-Lipschitz with respect the $\infty$-norm. 

Assume without loss of generality that $\|f\|_1=1$. By the reverse triangle inequality and by Lemma~\ref{lem:minvalue}, we only have to show that 
\[[-1,1]^{1+n}\ni x\mapsto \left\|\Delta^{-1}f\hm(x)\right\|_\infty\]
and
\[[-1,1]^{1+n}\ni x\mapsto \left\|\left(\diff_{x}f\hm P_0\right)^{\dagger}\Delta^2\right\|_{\infty,2}^{-1}=\min_{\substack{v\neq 0\\v\perp \ker A}}\frac{\|\Delta^{-2}\diff_xf[v]\|_\infty}{\|v\|_2} \]
are $1$-Lipschitz with respect the $\infty$-norm. Now, by the reverse triangle inequality again,
\[\left|\left\|\Delta^{-1}f\hm(y)\right\|_\infty-\left\|\Delta^{-1}f\hm(x)\right\|_\infty\right|\leq \left\|\Delta^{-1}f\hm(y)-\Delta^{-1}f\hm(x)\right\|_\infty\]
and
\begin{multline*}
\left|\min_{\substack{v\neq 0\\v\perp \ker A}}\frac{\|\Delta^{-2}\diff_xf[v]\|_\infty}{\|v\|_2}-\min_{\substack{v\neq 0\\v\perp \ker A}}\frac{\|\Delta^{-2}\diff_xf[v]\|_\infty}{\|v\|_2}\right|\\\leq \min_{\substack{v\neq 0\\v\perp \ker A}}\frac{\|\Delta^{-2}\diff_yf[v]-\Delta^{-2}\diff_xf[v]\|_\infty}{\|v\|_2}
\end{multline*}
Now, by Corollary~\ref{cor:onenormineqC}, $\Delta^{-1}f$ and $\Delta^{-2}\partial f[v]/\|v\|_\infty$ are $1$-Lipschitz with respect the $\infty$-norm. Therefore the proof is finished.
\end{proof}
\begin{proof}[Proof of Lemma~\ref{lem:minvalue}]
If $A$ is not surjective, then the claim follows since all terms are zero. Now, assume $A$ is surjective. We first prove the first equality and then the second one.

We have that
\[
    \|A^{\dagger}\|_{\infty,2}^{-1}=\left(\max_{w\neq 0}\frac{\|A^{\dagger}w\|_2}{\|w\|_\infty}\right)^{-1}=\min_{w\neq 0}\frac{\|w\|_\infty}{\|A^{\dagger}w\|_2}.
\]
Now, all values of $w\neq 0$ are of the form $Av$ with $v\neq 0$ and $v\perp \ker A$. Hence
\[
\min_{w\neq 0}\frac{\|w\|_\infty}{\|A^{\dagger}w\|_2}=\min_{\substack{v\neq 0\\v\perp\ker A}}\frac{\|Av\|_\infty}{\|A^{\dagger}Av\|_2}=\min_{\substack{v\neq 0\\v\perp\ker A}}\frac{\|Av\|_\infty}{\|v\|_2},
\]
since $A^\dagger A$ is the orthogonal projection onto the orthogonal complement of $\ker A$. This shows the first equality.

For the second equality, we only have to show that for a linear subspace $V\leq \bbR^{n}$ of dimension $n-q$, $\min_{v\in V,\,v\neq 0}\|Av\|_\infty/\|v\|_2\leq \min_{v\perp \ker A,\,v\neq 0}\|Av\|_{\infty}/\|v\|_2$. If $V\cap\ker A\neq 0$, then this is trivial. Otherwise, the orthogonal projection onto the orthogonal complement of $\ker A$, $A^\dagger A$, does not vanish on $V$, and so for all non-zero $v\in V$,
\[
\frac{\|Av\|_\infty}{\|v\|_2}=\frac{\|AA^\dagger Av\|_\infty}{\|A^\dagger Av\|_2}\frac{\|A^\dagger Av\|_2}{\|v\|_2}\leq \frac{\|A(A^\dagger Av)\|_\infty}{\|A^\dagger Av\|_2}.
\]
From here the desired inequality follows.
\end{proof}

We can show also a Condition Number Theorem, which gives a geometric interpretation of the condition number. We state it without a proof, since it is very similar to that of~\cite[Propositions~4.4 \& 4.6]{TCT-cube1}.

\begin{theo}\label{theo:conditionnumbertheorem}
Let $f\in \Pd[q]$ and
\[\Sigma_{n,\bfd}^R[q]:=\left\{g\in \Pd[q]\mid \exists\zeta\in\mcZ(g)\,:\,\|\zeta\|_\infty\leq R\right\}.\]
Then
\[
\frac{\|f\|_1}{\dist_1(\Delta^{-1}f,\Sigma_{n,\bfd}^R[q])}\leq \cond(f)\leq \frac{(1+\D)\|f\|_1}{\dist_1(\Delta^{-1}f,\Sigma_{n,\bfd}^R[q])}.
\]
\eproof
\end{theo}


\subsection{Condition-based Bound on Reach}

The higher derivative estimate induces a bound on the reach. However, this bound depends on both the condition number and the degree. The following bound does only depend on the condition number.

\begin{theo}\label{theo:condandreach}
Let $f\in\Pd[q]$ and $\zeta\in\mcZ(f)$. Then
\[\reach(\mcZ(f),\zeta)\geq \|\zeta\|_{\hinfty}/\max\left\{\D-2,\cond(f,\zeta)\right\}.\]
\end{theo}
\begin{cor}\label{cor:bddcondandreach}
Let $f\in\Pd[q]$. Then
\[\reach_R(\mcZ(f))\geq 1/\max\left\{\D-2,\cond_R(f)\right\}.\]
\end{cor}
\begin{remark}
The most interesting aspect of the bound is that the local reach of points ``near to infinity'' grows linearly with the distance to the origin.  The following two examples show what happens when the introduced condition number is not finite.
\begin{exam}[Two parallel lines]
The set formed by two parallel lines, e.g., $f=X(X-1)$, is projectively singular, since the lines ``intersect at infinity''. Thus $\lim_{R\to\infty}\cond_R(f)=\infty$. We can see that $\reach(\mcZ(f),(x,y))=1$, which clearly does not grow linearly at all.
\end{exam}
\begin{exam}
The parabola given by $f=Y-X^2$ intersects tangentially the line at infinity, and so $\lim_{R\to\infty}\cond_R(f)=\infty$. Since $\reach(\mcZ(f),(x,y))=\min\{1/2,|x|\}$,
\[\lim_{x\to\infty}\frac{\reach(f,(x,y))}{\|(x,y)\|_{\infty}}=0.\]
Differently to the case of the two lines, the reach of a parabola grows linearly as we go to infinity, although it does not grow linearly with respect to the distance to the origin. 
\end{exam}
\end{remark}
\begin{proof}[Proof of Theorem~\ref{theo:condandreach}]
We will apply Theorem~\ref{theo:kantorovichandreach}. Using submultiplicativity of the norm, we can easily see that $K(f,\zeta,r)$ is bounded by
\[\|D_\zeta f^{\dagger}\|\zeta\|_{\hinfty}^{\bfd-\bbone}\Delta^2\|_{\infty,2}\max_{z\in B(\zeta,r)}\left\|\|\zeta\|_{\hinfty}^{-\bfd+\bbone}\Delta^{-2}\diff_z^2f\right\|_{2,\infty}.\]
By the definition of the condition number and $\zeta\in\mcZ(f)$, we only have to focus on the right-hand side factor. Now,
\begin{align*}
\max_{z\in B(\zeta,r)}&\left\|\|\zeta\|_{\hinfty}^{-\bfd+\bbone}\Delta^{-2}\diff_z^2f\right\|_{2,\infty}\\&\leq \max_{z\in B(\zeta,r)}\left\|\|\zeta\|_{\hinfty}^{-\bfd+\bbone}\Delta^{-2}\diff_z^2f\right\|_{\infty,\infty}\\
&\leq \|\zeta\|_{\hinfty}^{-1}\max_{z\in B(\zeta,r)}\left(\frac{\|z\|_{\hinfty}}{\|\zeta\|_{\hinfty}}\right)^{\D-2}\left\|\|z\|_{\hinfty}^{-(\bfd-2\bbone)}\Delta^{-2}\diff_z^2f\right\|_{2,\infty}\\
&\leq \|f\|_1\|\zeta\|_{\hinfty}^{-1}\max_{z\in B(\zeta,r)}\left(\frac{\|z\|_{\hinfty}}{\|\zeta\|_{\hinfty}}\right)^{\D-2}
\end{align*}
where in the first line we use $\|\bbI\|_{2,\infty}=1$, in the second line we multiply and divide by $\|z\|^{\bfd-2\bbone}$ and apply submultiplicativity, and in the third line we use Corollary~\ref{cor:onenormineqB}.

Hence, under the assumption $r/\|\zeta\|_{\hinfty}<1/(\D-2)$, 
\begin{equation}
    K(f,\zeta,r)\leq \|\zeta\|_{\hinfty}^{-1}\cond(f,\zeta),
\end{equation}
since then $\|z\|_{\hinfty}\leq 1+\|\zeta\|_{\hinfty}/(\D-2)$. Now, for
\[r=\min\left\{\frac{\|\zeta\|_{\hinfty}}{\D-2},\frac{|\zeta\|_{\hinfty}}{\cond(f,\zeta)}\right\},\]
we have $K(f,\zeta,r)r<1$, and so we can apply Theorem~\ref{theo:kantorovichandreach} to finish the proof.
\end{proof}

\section[Random Bounds for the Reach: the Continuous Case]{Random Bounds for the Reach:\\the Continuous Case}

We establish bounds for the reach of the zero set of a zintzo random polynomial. To do this, we use the condition-based bound of the reach. The techniques in this section are essential for the next section, where we deal with integer bit random polynomials which are more technical to handle.

\subsection{Zintzo random polynomial tuples}

Zintzo random polynomials where introduced in~\cite{TCT-cube1conf} (cf.~\cite{TCT-cube1}) modelling a robust probabilistic model of polynomials. Here, we extend the notion to polynomial tuples and extend the results there to our condition number. We recall the following definitions needed for the definition of zintzo random polynomials.

\begin{defi}
Let $\fkx\in\bbR$ be a random variable and $p\geq 1$. Then:
\begin{enumerate}
    \item $\fkx$ has \emph{anticoncentration} if there is some $\rho>0$ such that for all $\varepsilon,x\in \bbR$, $\bbP(|\fkx-x|<\varepsilon)<2\rho \varepsilon.$ 
    The smallest such $\rho$ is called the \emph{anticoncentration constant} of $\fkx$.
    \item $\fkx$ is \emph{$p$-subexponential} (\emph{subgaussian} if $p\geq 0$) if there is some $L>0$ such that for all $t\geq L$, 
    $\bbP(|\fkx|\geq t)\leq\mathrm{e}^{-t^p/L^p}.$
    The smallest such constant is called the \emph{$p$-subexponential constant} of $\fkx$.
\end{enumerate}
\end{defi}

Now, we present the definition of zintzo random polynomials adapted to our setting. Note that support condition is similar to that of~\cite{renegar1987}, which is due to the chosen projectivization of $\bbR^n$.

\begin{defi}
\label{def:pzintzo}
Let $p\geq 1$ and $M_1,\ldots,M_q\subseteq \bbN^n$ finite sets such that for each $i$,
\begin{equation}\label{eq:renegarcond}
     R_{n,d_i}:=\bigcup_{k=0}^n\{(d_i-1)e_k+e_l\mid l\in \{0,\ldots,n\}\}\in M_i
\end{equation}
where by convention $e_0\in\bbZ^n$ is the zero vector and $\{e_k\}_{k=1}^n$ the canonical basis of $\bbN^n$.
A \emph{$p$-zintzo random polynomial $q$-tuple supported on $M_1,\ldots,M_q$} is a random polynomial tuple $\fkf\in\Pd[q]$ with
\[
  \fkf_i=\sum_{\alpha\in M_i}\fkf_{i,\alpha} X^\alpha\quad(i\in\{1,\ldots,q\})
\]
such that the coefficients $\fkf_\alpha$ are independent $p$-subexpoential random variables with the anti-concentration property.

Given a $p$-zintzo random polynomial tuple $\fkf\in\Pd[q]$, we define the following quantities:  
  \begin{enumerate}
\item The \emph{tail constants} of $\fkf$ given by
  \begin{equation}
    \label{eq:L-sg}
    L_{\fkf}:=\max_i\sum\nolimits_{\alpha\in M_i}L_{i,\alpha} ,
  \end{equation}
  where $L_{i,\alpha}$ is the $p$-subexponential constant of $\fkf_{i,\alpha}$.
\item The \emph{anti-concentration constant} of $\fkf$ given by
  \begin{equation}
    \label{eq:rho-acp}
    \rho_{\fkf}:=\max_{k,l}\rho_{i,(d_i-1)e_k+e_l}
  \end{equation}
  where $\rho_{i,\alpha}$ is the anti-concentration constant of $\fkf_{i,\alpha}$.
\end{enumerate}
\end{defi}

It is important to note that Kac polynomial tuples and polynomials with noise given by Kac random polynomials are zintzo random polynomial tuples. 

\begin{prop}
Let $\fkf\in\Pd[q]$ be a random polynomial tuple with supports $M_1,\ldots,M_q$  satisfying \eqref{eq:renegarcond}. The following holds:
\begin{enumerate}
\item[(g)] If the coefficients of $\fkf$ are independent normal variables of mean contained in $[-R,R]$ and variance at least $\sigma^2_0$ and at most $\sigma_1^2$, then
$L_{\fkf,i}=\max\{R,2\sigma_1\}\max_i|M_i|$ and $\rho_{\fkf,i}= (2\pi)^{-1/2}\sigma_0^{-1}$.
Moreover, if all coefficients are standard normal variables, then $\fkf$ is a $2$-zintzo random polynomial tuple such that for all $i$,
$L_{\fkf}=\sqrt{2}\max_i|M_i|$ and $\rho_{\fkf,i}= (2\pi)^{-1/2}$.
\item[(u)] If the coefficients of $\fkf$ are independent random variables uniformly distributed in intervals of length at least $\lambda$ contained in $[-R,R]$, then for all $p\geq 1$, $\fkf$ is a $p$-zintzo random polynomial tuple such that for all $i$, $L_{\fkf,i}\leq \max_i|M_i|R$ and $\rho_{\fkf,i}\leq \lambda^{-1}$.
In particular, if all coefficients are uniformly distributed in $[-1,1]$, then
$L_{\fkf,i}\leq \max_i|M_i|$ and $\rho_{\fkf,i}= 1/2$.\eproof
\end{enumerate}
\end{prop}

It is important to note that, arguing as in \cite[Proposition 7.11]{TCT-cube1}, for a $p$-zintzo random polynomial,
\begin{equation}\label{eq:lowerboundLrho}
L_{\fkf}\rho_{\fkf}\geq \frac{1}{4}\max_i |M_i|\geq \frac{1}{2}\binom{n+1}{2}\geq \frac{n^2}{4}.
\end{equation}




\subsection{Condition of a zintzo random polynomial}

The following theorems are standard. 

\begin{theo}\label{theo:condglobalbound_cont}
Let $\fkf\in\Pd[q]$ be a $p$-zintzo random polynomial tuple with supports $M_1,\ldots,M_q$. Then for $t\geq \mathrm{e}^{n+1}$,
\[\bbP(\cond_R(\fkf)\geq t)\leq 8R(n+1)^{\frac{p+2}{2p}n+\frac{q}{p}+1}\D^{q+2n}(4L_{\fkf}\rho_{\fkf})^{n+q}\frac{\ln^{\frac{n+q}{p}}t}{t}\]
\end{theo}
\begin{theo}\label{theo:condlocalbound_cont}
Let $\fkf\in\Pd[q]$ be a $p$-zintzo random polynomial tuple with supports $M_1,\ldots,M_q$ and $x\in \bbR^n$. Then for $t\geq \mathrm{e}^{n+1}$,
\[\bbP(\cond(\fkf,x)\geq t)\leq 4\|x\|_{\hinfty}(n+1)^{\frac{p+2}{2p}n+\frac{q}{p}}\D^{q+2n}(4L_{\fkf}\rho_{\fkf})^{n+q}\frac{\ln^{\frac{n+q}{p}}t}{t^{n+1}}\]
\end{theo}
\begin{proof}[Proof of Theorem~\ref{theo:condglobalbound_cont}]
Using Proposition~\ref{prop:condhomogeneous}, we will show that
\begin{equation}\label{eq:localtoglobal}
    \bbP(\cond_R(\fkf)\geq t)\leq (n+1)t^n\left(1+\frac{1}{2t}\right)^n\sup_{\|x\|_\infty \leq R}\bbP(\cond(\fkf,x)\geq t/2).
\end{equation}
Once this is shown, Theorem~\ref{theo:condlocalbound_cont} finishes the proof.

Let $\mcG_t\subset C:=\bbR_{> 0}^{n+1}\cap\partial [-1,1]^{n+1}\cap \{x\mid x_0\geq 1/R\}$ be a finite subset such that for all $x\in C$, $\dist_{\infty}(c,\mcG_t)\leq 1/t$. Now, by Proposition~\ref{prop:condhomogeneous}, 
\[\cond(\fkf)=\sup_{(x_0,x)\in X}\cond(f,x/x_0),\]
and, since $(x_0,x)\mapsto 1/\cond(f,x/x_0)$ is $1$-Lipschitz, we have that
\[\cond(\fkf)\leq 2\sup_{(g_0,g)\in\mcG_t}\cond(\fkf,g/g_0)\]
whenever $2\cond(\fkf)t\leq 1$. Thus $\cond(\fkf)\geq t$ implies  $$\sup_{(g_0,g)\in\mcG_t}\cond(\fkf,g/g_0)>t/2,$$ and so, by the implication bound,
\[
\bbP(\cond(\fkf)\geq t)\leq \bbP\left(\sup_{(g_0,g)\in\mcG_t}\cond(\fkf,g/g_0)\geq t/2\right),\]
and, by the union bound,
\[\bbP(\cond(\fkf)\geq t)
\leq |\mcG_t|\sup_{(x_0,x)\in X}\bbP(\cond(\fkf,x/x_0)\geq t/2).\]
Finally, we can achieve $|\mcG_t|\leq (n+1)t^n\left(1+\frac{1}{2t}\right)^n$ by putting the points in a uniform cubical grid, and $\sup_{(x_0,x)\in X}\bbP(\cond(\fkf,x/x_0)\geq t/2)=\sup_{\|x\|_\infty \leq R}\bbP(\cond(\fkf,x)\geq t/2)$. Thus \eqref{eq:localtoglobal} follows.
\end{proof}
\begin{proof}[Proof of Theorem~\ref{theo:condlocalbound_cont}]
Consider
\begin{equation*}
\cond(f,x,v):= \frac{\|f\|_1}{\max\left\{\left\|\Delta^{-1}\|x\|_{\hinfty}^{-\bfd}f(x)\right\|_\infty,\|v^*\|x\|_{\hinfty}^{\bbone-\bfd}\Delta^{-2}\diff_xf\|_2/\sqrt{n}\right\}}.
\end{equation*}
We can easily see, using Lemma~\ref{lem:minvalue}, that
\begin{equation}
\cond(f,x)\leq \max_{v\in \bbS^{q-1}}\cond(f,x,v)
\end{equation}
and, using standard norm inequalities and Corollary~\ref{cor:onenormineqB} as in the proof of the 2nd Lipschitz property of Theorem~\ref{theo:condnumberproperty}, that the map
\begin{equation}
    \bbS^{q-1}\ni v\mapsto 1/\cond(f,x,v)
\end{equation}
is $1$-Lipschitz with respect both $\|~\|_2$ and the geodesic distance. 

Now, let $\fkv_*\in\bbS^q$, be the random vector such that
\[\max_{v\in \bbS^{q-1}}\cond(\fkf,x,v)=\cond(\fkf,x,\fkv_*).\]
Then, if $\max_{v\in \bbS^{q-1}}\cond(\fkf,x,v)\geq t$, we have have that for all $v\in B_{\bbS}(\fkv_*,1/t)\subset\bbS^{q-1}$, $\cond(\fkf,x,v)\geq t/2$. Hence $\max_{v\in \bbS^{q-1}}\cond(\fkf,x,v)\geq t$ implies that
\[
\bbP_{\fkv\in\bbS^{q-1}}\left(\cond(\fkf,x,\fkv)\geq t/2\right)\geq \frac{\vol B_{\bbS}(\fkv_*,1/t)}{\vol\bbS^q}\geq \frac{1}{\sqrt{2\pi q}t^{q-1}}
\]
where the last inequality follows from~\cite[Lemma 2.34]{conditionbook}. Thus
\begin{align}\label{eq:tangenttolocal}
\begin{split}
    \bbP_{\fkf}(\cond(f,x)&\geq t)\\&\leq \bbP(\max_{v\in \bbS^{q-1}}\cond(f,x,v)\geq t)\\
    &\leq \bbP_{\fkf}\left(\bbP_{\fkv\in\bbS^{q-1}}(\cond(f,x,v)\geq  t/2)\geq \frac{1}{\sqrt{2\pi q}t^{q-1}}\right)\\
    &\leq \sqrt{2\pi q}t^{q-1}\bbE_{\fkf}\bbP_{\fkv\in\bbS^{q-1}}(\cond(f,x,v)\geq t/2)\\
    &\leq \sqrt{2\pi q}t^{q-1}\bbE_{\fkv\in\bbS^{q-1}}\bbP_{\fkf}(\cond(f,x,v)\geq t/2)\\
    &\leq \sqrt{2\pi q}t^{q-1}\sup_{v\in\bbS^{q-1}}\bbP_{\fkf}(\cond(f,x,v)\geq t/2)
\end{split}
\end{align}
where in the first two inequalities we use the implication bound; in the third one, Markov's Inequality; and in the fourth one, the Fubini-Tonelli theorem.

All this reduces our work to estimating
\[\bbP_{\fkf}(\cond(f,x,v)\geq t/2).\]
By the union bound, the implication bound and \cite[Proposition 7.4]{TCT-cube1}, we have that for $u\geq L_{\fkf}$,
\begin{multline*}
    \bbP_{\fkf}(\cond(f,x,v)\geq t/2)\leq \left(\sum_{i}|M_i|\right)\mathrm{e}^{-u^p/L_{\fkf}^p}\\
    +\bbP_{\fkf}\left(\max\{\|\|x\|_{\hinfty}^{-\bfd}\Delta^{-1}f(x)\|_\infty,\|v^*\|x\|_{\hinfty}^{\bbone-\bfd}\Delta^{-2}\diff_x f\|_2/\sqrt{n}\}\leq u/t \right),
\end{multline*}
where the implication used is that either the numerator of $\cond(f,x,v)$ is $\geq u$ or the denominator $\leq u/t$.

To bound the last summand, we will use \cite[Proposition 7.7]{TCT-cube1}. This proposition, an explicit refinement of~\cite[Theorem~1.1]{RV-1}, states that given a random vector $\fkx\in \bbR^N$ with independent coordinates with anticoncentration constants $\leq \rho$, a surjective matrix $A\in\bbR^{k\times N}$ and $U\subseteq \bbR^k$, then
\[\bbP(A\fkx\in U)\leq \vol(U)(\sqrt{2}\rho)^k/\sqrt{\det(AA^*)}.\]
By the Cauchy-Binet formula, we only have to find appropiate submatrices of $A$ to obtain lower bounds on $\sqrt{\det(AA^*)}$. 

Now, to apply the above result, we consider the linear projection,
\begin{equation}\label{eq:linproj}
    f\mapsto \begin{pmatrix}
        \|x\|_{\hinfty}^{-\bfd}\Delta^{-1}f(x)\\
        (\|x\|_{\hinfty}^{\bbone-\bfd}\Delta^{-2}\diff_x f)^* v/\sqrt{n}
    \end{pmatrix}\in \bbR^{q+n}
\end{equation}
with $U=u/t [-1,1]^q\times \bbB^q$, where $\bbB^q$ is the $q$-ball. Stirling's approximation~\cite[(2.14)]{conditionbook} shows that 
\[
\vol U \leq \frac{1}{\sqrt{2\pi}}\frac{2^n (\mathrm{e}\pi)^{q/2}u^{q+n}}{q^{q/2} t^{q+n}}.
\]
If $1=\|x\|_{\hinfty}$, then we obtain the following submatrix of \eqref{eq:linproj}
\begin{equation}\label{eq:matrixlinprojA}
\begin{pmatrix}
    \Delta^{-1} & x_1\Delta^{-1}&\cdots&x_n\Delta^{-1}\\
    &v^*\Delta^{-2}/\sqrt{n}&&\\
    &&\ddots&\\
    &&&v^*\Delta^{-2}/\sqrt{n}
\end{pmatrix}
\end{equation}
by looking at the monomials $1,X_1,\ldots,X_n$. In the above matrix, each column-block corresponds to one of the monomials; each row block correspond to the polynomials. If $|x_i|=\|x\|_{\hinfty}$, then we obtain
{\tiny\begin{equation}\label{eq:matrixlinprojB}
\begin{pmatrix}
    \Delta^{-1} &\Delta^{-1}|x_i|^{-1}&x_1|x_i|^{-1}\Delta^{-1}&\cdots&x_n|x_i|^{-1}\Delta^{-1}\\
    &-|x_i|^{-1}v^*\Delta^{-2}/\sqrt{n}&-|x_i|^{-1}x_1v^*\Delta^{-2}/\sqrt{n}&\cdots&-|x_i|^{-1}x_nv^*\Delta^{-2}/\sqrt{n}\\
    &&v^*\Delta^{-2}/\sqrt{n}&&\\
    &&&\ddots&\\
    &&&&v^*\Delta^{-2}/\sqrt{n}
\end{pmatrix}
\end{equation}}
after adding to the second row block multiples of rows in the first row block and looking at the monomials $X_k^{d_i}$, $X_k^{d_i-1}
$,$X^{d_i-1}_k X_1$,$\ldots,$ $X^{d_i-1}_kX_n$, where we omit $X_{k}^{d_i-1}X_k$ in the latter list. 
Hence, in both cases, after an easy computation and applying the Cauchy-Binet formula, we get
\begin{equation}
    \sqrt{\det AA^*}\geq \sqrt{q}\|x\|_{\hinfty}^{-1}n^{-\frac{n}{2}}\D^{-q-2n}.
\end{equation}
Therefore $\bbP(\cond(\fkf,x,v)\geq t/2)$ is upper bounded by
\[
\left(\sum_{i}|M_i|\right)\mathrm{e}^{-u^p/L_{\fkf}^p} + \frac{\|x\|_{\hinfty}}{\sqrt{2\pi}}\frac{2^n (\mathrm{e}\pi)^{q/2}n^{\frac{n}{2}}\D^{q+2n}u^{q+n}\rho_{\fkf}^{n+q}}{q^{(q+1)/2} t^{q+n}},
\]
and so $\bbP(\cond(\fkf,x)\geq t)$ by
\[
\sqrt{2\pi q}\left(\sum_{i}|M_i|\right)\mathrm{e}^{-u^p/L_{\fkf}^p} + \|x\|_{\hinfty}{}\frac{2^n (\mathrm{e}\pi)^{q/2}n^{\frac{n}{2}}\D^{q+2n}u^{q+n}\rho_{\fkf}^{n+q}}{q^{(q-1)/2} t^{1+n}}.
\]
Taking $u=L_{\fkf}((n+1)\ln t)^{1/p}$ and estimating gives the bound. We use that $|M_i|\leq (2\D)^n$ among other easy estimates.
\end{proof}

\subsection{Bounds on the Reach: Proof of Theorem~\ref{theo:probcaseboundreach_contunif}}

Theorem~\ref{theo:probcaseboundreach_contunif} is a particular case of the following theorem.

\begin{theo}\label{theo:reachtailbound_cont}
Let $\fkf\in\Pd[q]$ be a $p$-zintzo random polynomial tuple with supports $M_1,\ldots,M_q$. Then for $\varepsilon\in (0,\min\{1/\D,1/\mathrm{e}^{-(n+1)}\}]$, $\bbP(\reach_R(\mcZ(f))\leq \varepsilon)$ is at most
\[
 8R(n+1)^{\frac{p+2}{2p}n+\frac{q}{p}+1}\D^{q+2n}(4L_{\fkf}\rho_{\fkf})^{n+q}\,\varepsilon \ln^{\frac{n+q}{p}}\left(\frac{1}{\varepsilon}\right)\]
\end{theo}
\begin{proof}[Proof of Theorem~\ref{theo:reachtailbound_cont}]
This is Corollary \ref{cor:bddcondandreach} combined with Theorem \ref{theo:condglobalbound_cont}.
\end{proof}

\section{Integer Bounds for the Reach}

\subsection{Worst-case: Proof of Theorem~\ref{theo:worstcaseboundreach}}



\begin{proof}[Proof of Theorem~\ref{theo:worstcaseboundreach}] By~\eqref{eq:condhom} and arguing as in the proof of Theorem~\ref{theo:condglobalbound_cont}, we can see that $\cond_R(f)$ is bounded by
\[
\D^2\|f\|_1/\min_{z,\,v\in\bbS^{q-1}}\max\{\|f\hm(z)\|_{\infty},\|v^*\diff_zf\hm P_0\|_2\}
\]
where $z\in\partial [-1,1]^{n+1}$ satisfies that for each $i$, $-Rz_0\leq z_i\leq Rz_0$. Now, by assumption, $\D^2\|f\|_1\leq \D^22^{\tau}$. Hence, we only need to handle the denominator. For this, we will use \cite[Theorem~1]{jeronimoperruccitsigaridas2013} by Jeronimo, Perrucci and Tsigaridas.

Now, consider the semialgebraic set $S\subseteq \bbR^{1+q+2n}$ given by
\[
\pm z_i\leq Rz_0,\,\|v\|_2=1,\,t\geq \pm f_i\hm(x),\,t^2\geq \|w\|_2^2,\,w=\diff_zf\hm
\]
where $(x,v,w,t)\in\bbR^n\times\bbR^q\times \bbR^n\times \bbR$. This semialgebraic set is described by $2+4n$ polynomials of degree at most $\D$ and coefficients of bit size at most $\max\{\log R,\tau\}$. Moreover,
\[
\min_{(x,v,w,t)\in S}t=\min_{z,\,v\in\bbS^{q-1}}\max\{\|f\hm(z)\|_{\infty},\|v^*\diff_zf\hm P_0\|_2\}.
\]
Hence, by \cite[Theorem~1]{jeronimoperruccitsigaridas2013}, $\cond_R(f)=\infty$ or
\[
\log \cond_R(f)\leq 4n(2\D)^{1+q+2n}(5+\tau+\log R+6n\log \D)+2\log\D+\tau.
\]
Using Theorem~\ref{theo:condandreach}, the desired bound follows.
\end{proof}

\begin{remark}
Unfortunately, the usual formulations of the reach do not allow for the application of \cite[Theorem~1]{jeronimoperruccitsigaridas2013}.
\end{remark}

\subsection{Random bit polynomials}

We extend the definition of random bit polynomials from~\cite{ETCT-descartes} to the multivariate setting in which we are working.

\begin{defi}
Let $M_1,\ldots,M_q\subseteq \bbN^n$ be finite sets satisfying~\eqref{eq:renegarcond}. A \emph{random bit polynomial $q$ tuple with bit-size $\tau$} is a random polynomial tuple $\fkf\in\Pd[q]$ with $
  \fkf_i=\sum_{\alpha\in M_i}\fkf_{i,\alpha} X^\alpha$ for $i\in\{1,\ldots,q\}$,
such that the coefficients $\fkf_\alpha$ are independent integers of absolute value at most $2^{\tau}$. 

Given such a random polynomial $\fkf$, its \emph{weight} is
\begin{equation}
    w(\fkf):=\min\left\{\bbP(\fkf_{i,\alpha}=c)\mid c\in\bbZ,\,(i,\alpha)\in \cup_{i=1}^n \{i\}\times R_{n,d_i}\right\};
\end{equation}
and its \emph{uniformity} is
\begin{equation}
    u(\fkf):=\ln \left( (1+2^{\tau+1}) \, w(\fkf)  \right) .
\end{equation}
\end{defi}

\begin{exam}
A \emph{uniform random bit polynomial $q$-tuple of bit-size $\tau$} is a random polynomial $q$-tuple $\fkf\in\Pd[q]$ whose coefficients are independent uniformly distributed integers in $[-2^{\tau},2^{\tau}]$. Such a polynomial is a random bit polynomial with uniformity $1$.
\end{exam}
\begin{remark}
As shown in \cite{ETCT-descartes}, this random model has the advantage that it is robust under change of support, selection of signs and selection of exact bit-size among others.
\end{remark}

\subsection{Condition of random bit polynomials}

We now prove the discrete analogue of Theorems~\ref{theo:condglobalbound_cont} and~\ref{theo:condlocalbound_cont}. 

\begin{theo}\label{theo:condglobalbound_disc}
Let $\fkf\in\Pd[q]$ be a random bit polynomial tuple of bit-size $\tau$. Then for $t\in [n, q(n+1)\max_i|M_i|\,2^{\tau-2}]$,
\[\bbP(\cond_R(\fkf)\geq t)\leq 20 R n^{\frac{n}{2}+1}\left(\sqrt{2}\max_i|M_i|\right)^{q+n}\D^{q+2n}\mathrm{e}^{u(\fkf)}\frac{1}{t}.\]
\end{theo}
\begin{theo}\label{theo:condlocalbound_disc}
Let $\fkf\in\Pd[q]$ be be a random bit polynomial tuple of bit-size $\tau$ and $x\in \bbR^n$. Then for $t\leq q(n+1)\max_i|M_i|\,2^{\tau-1}$,
\[\bbP(\cond(\fkf,x)\geq t)\leq 6\|x\|_{\hinfty}n^{\frac{n}{2}}\left(\sqrt{2}\max_i|M_i|\right)^{q+n}\D^{q+2n}\mathrm{e}^{u(\fkf)}\frac{1}{t^{n+1}}.\]
\end{theo}
\begin{remark}
Note that Theorems~\ref{theo:condglobalbound_disc} and \ref{theo:condlocalbound_disc} are meaningful only for $\tau$ large enough.
\end{remark}
\begin{proof}[Proof of Theorem~\ref{theo:condglobalbound_disc}]
This is~\eqref{eq:localtoglobal} with Theorem~\ref{theo:condlocalbound_disc}.
\end{proof}
\begin{proof}[Proof of Theorem~\ref{theo:condlocalbound_disc}]
The proof follows the same steps as that of Theorem~\ref{theo:condlocalbound_cont}. However, it differs in two points: (1) There is no need to control the norm, since $\|f\|_1\leq \max|M_i|2^{\tau}$. Hence
\begin{multline*}
  \bbP_{\fkf}(\cond(\fkf,x,v)\\\leq \bbP_{\fkf}\left(\max\{\|\|x\|_{\hinfty}^{-\bfd}\Delta^{-1}f(x)\|_\infty,\|v^*\|x\|_{\hinfty}^{\bbone-\bfd}\Delta^{-2}\diff_x f\|_2\}\leq u/t \right)  \\
  \leq \bbP_{\fkf}\left(\max\{\|\|x\|_{\hinfty}^{-\bfd}\Delta^{-1}f(x)\|_\infty,\|v^*\|x\|_{\hinfty}^{\bbone-\bfd}\Delta^{-2}\diff_x f\|_\infty\}\leq u/t \right) 
\end{multline*}
where $u=\max|M_i|\,2^{\tau}$. (2) Where we applied \cite[Proposition 7.7]{TCT-cube1}, we apply now \cite[Proposition~2.7]{ETCT-descartes} to the map \eqref{eq:linproj}. This proposition states that for $\fkx\in\bbZ^N$ random such that all probability weights are $\leq w$, $A\in \bbR^{k\times N}$ and $b\in\bbR^k$, we have that for $\varepsilon \in [\|A\|_{\infty,\infty},\infty)$,
\[
\bbP(\|A\fkx+b\|_{\infty}\leq \varepsilon)\leq 2\left(2\sqrt{2}w\varepsilon\right)^k/\sqrt{\det AA^*}.
\]
Moreover, conditioning on all random coefficients that are not in the $R_{n,d_i}$ (see~\eqref{eq:renegarcond}), we can assume without loss of generality that $M_i=R_{n,d_i}$. Then the matrices of the affine map~\eqref{eq:linproj} are still~\eqref{eq:matrixlinprojA} and~\eqref{eq:matrixlinprojB}. In both cases, we have $\|A\|_{\infty,\infty}\leq q(n+1)$. Hence $\bbP(\cond(\fkf,x,v)\geq t/2)$ is upper bounded by
\[
2q^{-\frac{1}{2}}\|x\|_{\hinfty}n^{\frac{n}{2}}\D^{q+2n}\left(2\sqrt{2}w(\fkf)\max_i|M_i|2^{\tau}\right)^{q+n}t^{-(q+n)}
\]
whenever $t\leq q(n+1)\max_i|M_i|\,2^{\tau}$. Now, $2w(\fkf)2^\tau\leq\mathrm{e}^{u(\fkf)}$ and \eqref{eq:tangenttolocal} gives the desired bound.
\end{proof}

\subsection{Probabilistic case: Proof of Theorem~\ref{theo:probcaseboundreach}}

Theorem~\ref{theo:probcaseboundreach} is a particular case of the following theorem.

\begin{theo}\label{theo:reach_randombit}
Let $\fkf\in\Pd[q]$ be a random bit polynomial tuple of bit-size $\tau$. Then for $\varepsilon \in [2/\left(q(n+1)\max_i|M_i|\,2^{\tau}\right),\min\{1/\D,1/n\}]$,
\[\bbP(\reach_R(\mcZ(\fkf))\leq \varepsilon)\leq  20 R n^{\frac{n}{2}+1}\left(\sqrt{2}\max_i|M_i|\right)^{q+n}\D^{q+2n}\mathrm{e}^{u(\fkf)}\varepsilon.\]
\end{theo}
\begin{remark}
Note that Theorem~\ref{theo:reach_randombit} is meaningful only for $\tau$ large enough.
\end{remark}
\begin{proof}[Proof of Theorem~\ref{theo:reach_randombit}]
We just apply Theorems~\ref{theo:condglobalbound_disc} and~\ref{theo:condandreach}.
\end{proof}


%% file: main.bbl

\begin{thebibliography}{28}


\ifx \showCODEN    \undefined \def \showCODEN     #1{\unskip}     \fi
\ifx \showDOI      \undefined \def \showDOI       #1{#1}\fi
\ifx \showISBNx    \undefined \def \showISBNx     #1{\unskip}     \fi
\ifx \showISBNxiii \undefined \def \showISBNxiii  #1{\unskip}     \fi
\ifx \showISSN     \undefined \def \showISSN      #1{\unskip}     \fi
\ifx \showLCCN     \undefined \def \showLCCN      #1{\unskip}     \fi
\ifx \shownote     \undefined \def \shownote      #1{#1}          \fi
\ifx \showarticletitle \undefined \def \showarticletitle #1{#1}   \fi
\ifx \showURL      \undefined \def \showURL       {\relax}        \fi
\providecommand\bibfield[2]{#2}
\providecommand\bibinfo[2]{#2}
\providecommand\natexlab[1]{#1}
\providecommand\showeprint[2][]{arXiv:#2}

\bibitem[Aamari et~al\mbox{.}(2019)]%
        {aamarikimfredericmichelrinaldowasserman2019}
\bibfield{author}{\bibinfo{person}{Eddie Aamari}, \bibinfo{person}{Jisu Kim},
  \bibinfo{person}{Fr\'{e}d\'{e}ric Chazal}, \bibinfo{person}{Bertrand Michel},
  \bibinfo{person}{Alessandro Rinaldo}, {and} \bibinfo{person}{Larry
  Wasserman}.} \bibinfo{year}{2019}\natexlab{}.
\newblock \showarticletitle{Estimating the reach of a manifold}.
\newblock \bibinfo{journal}{\emph{Electron. J. Stat.}} \bibinfo{volume}{13},
  \bibinfo{number}{1} (\bibinfo{year}{2019}), \bibinfo{pages}{1359--1399}.
\newblock
\showISSN{1935-7524}
\urldef\tempurl%
\url{https://doi.org/10.1214/19-ejs1551}
\showDOI{\tempurl}


\bibitem[Breiding et~al\mbox{.}(2018)]%
        {breidingkaivniksturmfels2018}
\bibfield{author}{\bibinfo{person}{Paul Breiding}, \bibinfo{person}{Sara
  Kali\v{s}nik}, \bibinfo{person}{Bernd Sturmfels}, {and}
  \bibinfo{person}{Madeleine Weinstein}.} \bibinfo{year}{2018}\natexlab{}.
\newblock \showarticletitle{Learning algebraic varieties from samples}.
\newblock \bibinfo{journal}{\emph{Rev. Mat. Complut.}} \bibinfo{volume}{31},
  \bibinfo{number}{3} (\bibinfo{year}{2018}), \bibinfo{pages}{545--593}.
\newblock
\showISSN{1139-1138,1988-2807}
\urldef\tempurl%
\url{https://doi.org/10.1007/s13163-018-0273-6}
\showDOI{\tempurl}


\bibitem[B\"{u}rgisser and Cucker(2013)]%
        {conditionbook}
\bibfield{author}{\bibinfo{person}{Peter B\"{u}rgisser} {and}
  \bibinfo{person}{Felipe Cucker}.} \bibinfo{year}{2013}\natexlab{}.
\newblock \bibinfo{booktitle}{\emph{Condition}}. \bibinfo{series}{Grundlehren
  der mathematischen Wissenschaften [Fundamental Principles of Mathematical
  Sciences]}, Vol.~\bibinfo{volume}{349}.
\newblock \bibinfo{publisher}{Springer}, \bibinfo{address}{Heidelberg}.
  xxxii+554 pages.
\newblock
\showISBNx{978-3-642-38895-8; 978-3-642-38896-5}
\urldef\tempurl%
\url{https://doi.org/10.1007/978-3-642-38896-5}
\showDOI{\tempurl}
\newblock
\shownote{The geometry of numerical algorithms}.


\bibitem[B\"{u}rgisser et~al\mbox{.}(2019)]%
        {BCL19}
\bibfield{author}{\bibinfo{person}{P. B\"{u}rgisser}, \bibinfo{person}{F.
  Cucker}, {and} \bibinfo{person}{P. Lairez}.} \bibinfo{year}{2019}\natexlab{}.
\newblock \showarticletitle{Computing the homology of basic semialgebraic sets
  in weak exponential time}.
\newblock \bibinfo{journal}{\emph{J. ACM}} \bibinfo{volume}{66},
  \bibinfo{number}{1} (\bibinfo{year}{2019}), \bibinfo{pages}{Art. 5, 30}.
\newblock
\showISSN{0004-5411}
\urldef\tempurl%
\url{https://doi.org/10.1145/3275242}
\showDOI{\tempurl}
\newblock
\shownote{[Publication date initially given as 2018]}.


\bibitem[B{\"u}rgisser et~al\mbox{.}(2020)]%
        {BCTC1}
\bibfield{author}{\bibinfo{person}{P. B{\"u}rgisser}, \bibinfo{person}{F.
  Cucker}, {and} \bibinfo{person}{J. Tonelli-Cueto}.}
  \bibinfo{year}{2020}\natexlab{}.
\newblock \showarticletitle{{Computing the Homology of Semialgebraic Sets. I:
  Lax Formulas}}.
\newblock \bibinfo{journal}{\emph{Foundations of Computational Mathematics}}
  \bibinfo{volume}{20}, \bibinfo{number}{1} (\bibinfo{year}{2020}),
  \bibinfo{pages}{71--118}.
\newblock
\showISSN{1615-3383}
\urldef\tempurl%
\url{https://doi.org/10.1007/s10208-019-09418-y}
\showDOI{\tempurl}
\newblock
\shownote{On-line from May of 2019}.


\bibitem[B{\"u}rgisser et~al\mbox{.}(2021)]%
        {BCTC2}
\bibfield{author}{\bibinfo{person}{P. B{\"u}rgisser}, \bibinfo{person}{F.
  Cucker}, {and} \bibinfo{person}{J. Tonelli-Cueto}.}
  \bibinfo{year}{2021}\natexlab{}.
\newblock \showarticletitle{{Computing the Homology of Semialgebraic Sets. II:
  General formulas}}.
\newblock \bibinfo{journal}{\emph{Foundations of Computational Mathematics}}
  \bibinfo{volume}{21}, \bibinfo{number}{5} (\bibinfo{date}{1}
  \bibinfo{year}{2021}), \bibinfo{pages}{1279--1316}.
\newblock
\showISSN{1615-3383}
\urldef\tempurl%
\url{https://doi.org/10.1007/s10208-020-09483-8}
\showDOI{\tempurl}
\newblock
\shownote{On-line from January of 2021. arXiv: 1903.10710}.


\bibitem[Cucker et~al\mbox{.}(2018)]%
        {CKS16}
\bibfield{author}{\bibinfo{person}{F. Cucker}, \bibinfo{person}{T. Krick},
  {and} \bibinfo{person}{M. Shub}.} \bibinfo{year}{2018}\natexlab{}.
\newblock \showarticletitle{Computing the homology of real projective sets}.
\newblock \bibinfo{journal}{\emph{Found. Comput. Math.}} \bibinfo{volume}{18},
  \bibinfo{number}{4} (\bibinfo{year}{2018}), \bibinfo{pages}{929--970}.
\newblock
\showISSN{1615-3375,1615-3383}
\urldef\tempurl%
\url{https://doi.org/10.1007/s10208-017-9358-8}
\showDOI{\tempurl}


\bibitem[Dedieu(2006)]%
        {dedieubook}
\bibfield{author}{\bibinfo{person}{J.-P. Dedieu}.}
  \bibinfo{year}{2006}\natexlab{}.
\newblock \bibinfo{booktitle}{\emph{Points fixes, z\'{e}ros et la m\'{e}thode
  de {N}ewton}}. \bibinfo{series}{Math\'{e}matiques \& Applications (Berlin)
  [Mathematics \& Applications]}, Vol.~\bibinfo{volume}{54}.
\newblock \bibinfo{publisher}{Springer}, \bibinfo{address}{Berlin}. xii+196
  pages.
\newblock
\showISBNx{978-3-540-30995-6; 3-540-30995-0}


\bibitem[Di~Rocco et~al\mbox{.}(2023)]%
        {diroccoedwardseklundgafverthauenstein2023}
\bibfield{author}{\bibinfo{person}{Sandra Di~Rocco}, \bibinfo{person}{Parker~B.
  Edwards}, \bibinfo{person}{David Eklund}, \bibinfo{person}{Oliver
  G\"{a}fvert}, {and} \bibinfo{person}{Jonathan~D. Hauenstein}.}
  \bibinfo{year}{2023}\natexlab{}.
\newblock \showarticletitle{Computing geometric feature sizes for algebraic
  manifolds}.
\newblock \bibinfo{journal}{\emph{SIAM J. Appl. Algebra Geom.}}
  \bibinfo{volume}{7}, \bibinfo{number}{4} (\bibinfo{year}{2023}),
  \bibinfo{pages}{716--741}.
\newblock
\showISSN{2470-6566}
\urldef\tempurl%
\url{https://doi.org/10.1137/22M1522656}
\showDOI{\tempurl}


\bibitem[Di~Rocco et~al\mbox{.}(2022)]%
        {diroccoeklundgafvert2022}
\bibfield{author}{\bibinfo{person}{Sandra Di~Rocco}, \bibinfo{person}{David
  Eklund}, {and} \bibinfo{person}{Oliver G\"{a}fvert}.}
  \bibinfo{year}{2022}\natexlab{}.
\newblock \showarticletitle{Sampling and homology via bottlenecks}.
\newblock \bibinfo{journal}{\emph{Math. Comp.}} \bibinfo{volume}{91},
  \bibinfo{number}{338} (\bibinfo{year}{2022}), \bibinfo{pages}{2969--2995}.
\newblock
\showISSN{0025-5718,1088-6842}
\urldef\tempurl%
\url{https://doi.org/10.1090/mcom/3757}
\showDOI{\tempurl}


\bibitem[Di~Rocco et~al\mbox{.}(2020)]%
        {diroccoeklundweinstein2020}
\bibfield{author}{\bibinfo{person}{Sandra Di~Rocco}, \bibinfo{person}{David
  Eklund}, {and} \bibinfo{person}{Madeleine Weinstein}.}
  \bibinfo{year}{2020}\natexlab{}.
\newblock \showarticletitle{The bottleneck degree of algebraic varieties}.
\newblock \bibinfo{journal}{\emph{SIAM J. Appl. Algebra Geom.}}
  \bibinfo{volume}{4}, \bibinfo{number}{1} (\bibinfo{year}{2020}),
  \bibinfo{pages}{227--253}.
\newblock
\showISSN{2470-6566}
\urldef\tempurl%
\url{https://doi.org/10.1137/19M1265776}
\showDOI{\tempurl}


\bibitem[Eklund(2023)]%
        {eklund2023}
\bibfield{author}{\bibinfo{person}{David Eklund}.}
  \bibinfo{year}{2023}\natexlab{}.
\newblock \showarticletitle{The numerical algebraic geometry of bottlenecks}.
\newblock \bibinfo{journal}{\emph{Adv. in Appl. Math.}}  \bibinfo{volume}{142}
  (\bibinfo{year}{2023}), \bibinfo{pages}{Paper No. 102416, 20}.
\newblock
\showISSN{0196-8858,1090-2074}
\urldef\tempurl%
\url{https://doi.org/10.1016/j.aam.2022.102416}
\showDOI{\tempurl}


\bibitem[Emiris et~al\mbox{.}(2020)]%
        {emirismourraintsigaridas2020}
\bibfield{author}{\bibinfo{person}{Ioannis Emiris}, \bibinfo{person}{Bernard
  Mourrain}, {and} \bibinfo{person}{Elias Tsigaridas}.}
  \bibinfo{year}{2020}\natexlab{}.
\newblock \showarticletitle{Separation bounds for polynomial systems}.
\newblock \bibinfo{journal}{\emph{J. Symbolic Comput.}}  \bibinfo{volume}{101}
  (\bibinfo{year}{2020}), \bibinfo{pages}{128--151}.
\newblock
\showISSN{0747-7171,1095-855X}
\urldef\tempurl%
\url{https://doi.org/10.1016/j.jsc.2019.07.001}
\showDOI{\tempurl}


\bibitem[Emiris et~al\mbox{.}(2010)]%
        {emirismourraintsigaridas2010}
\bibfield{author}{\bibinfo{person}{Ioannis~Z. Emiris}, \bibinfo{person}{Bernard
  Mourrain}, {and} \bibinfo{person}{Elias~P. Tsigaridas}.}
  \bibinfo{year}{2010}\natexlab{}.
\newblock \showarticletitle{The {DMM} bound: multivariate (aggregate)
  separation bounds}. In \bibinfo{booktitle}{\emph{I{SSAC} 2010---{P}roceedings
  of the 2010 {I}nternational {S}ymposium on {S}ymbolic and {A}lgebraic
  {C}omputation}}. \bibinfo{publisher}{ACM}, \bibinfo{address}{New York},
  \bibinfo{pages}{243--250}.
\newblock
\showISBNx{978-1-4503-0150-3}
\urldef\tempurl%
\url{https://doi.org/10.1145/1837934.1837981}
\showDOI{\tempurl}


\bibitem[Erg\"{u}r et~al\mbox{.}(2022)]%
        {ETCT-descartes}
\bibfield{author}{\bibinfo{person}{Alperen Erg\"{u}r},
  \bibinfo{person}{Josu\'{e} Tonelli-Cueto}, {and} \bibinfo{person}{Elias
  Tsigaridas}.} \bibinfo{year}{[2022] \copyright 2022}\natexlab{}.
\newblock \showarticletitle{Beyond worst-case analysis for root isolation
  algorithms}. In \bibinfo{booktitle}{\emph{I{SSAC} '22---{P}roceedings of the
  2022 {I}nternational {S}ymposium on {S}ymbolic and {A}lgebraic
  {C}omputation}}. \bibinfo{publisher}{ACM}, \bibinfo{address}{New York},
  \bibinfo{pages}{139--148}.
\newblock


\bibitem[Erg\"{u}r et~al\mbox{.}(2019)]%
        {EPR-probcond1}
\bibfield{author}{\bibinfo{person}{Alperen~A. Erg\"{u}r},
  \bibinfo{person}{Grigoris Paouris}, {and} \bibinfo{person}{J.~Maurice
  Rojas}.} \bibinfo{year}{2019}\natexlab{}.
\newblock \showarticletitle{Probabilistic condition number estimates for real
  polynomial systems {I}: {A} broader family of distributions}.
\newblock \bibinfo{journal}{\emph{Found. Comput. Math.}} \bibinfo{volume}{19},
  \bibinfo{number}{1} (\bibinfo{year}{2019}), \bibinfo{pages}{131--157}.
\newblock
\showISSN{1615-3375,1615-3383}
\urldef\tempurl%
\url{https://doi.org/10.1007/s10208-018-9380-5}
\showDOI{\tempurl}


\bibitem[Erg\"{u}r et~al\mbox{.}(2021)]%
        {EPR-probcond2}
\bibfield{author}{\bibinfo{person}{Alperen~A. Erg\"{u}r},
  \bibinfo{person}{Grigoris Paouris}, {and} \bibinfo{person}{J.~Maurice
  Rojas}.} \bibinfo{year}{2021}\natexlab{}.
\newblock \showarticletitle{Smoothed analysis for the condition number of
  structured real polynomial systems}.
\newblock \bibinfo{journal}{\emph{Math. Comp.}} \bibinfo{volume}{90},
  \bibinfo{number}{331} (\bibinfo{year}{2021}), \bibinfo{pages}{2161--2184}.
\newblock
\showISSN{0025-5718,1088-6842}
\urldef\tempurl%
\url{https://doi.org/10.1090/mcom/3647}
\showDOI{\tempurl}


\bibitem[Federer(1959)]%
        {federer1959}
\bibfield{author}{\bibinfo{person}{H. Federer}.}
  \bibinfo{year}{1959}\natexlab{}.
\newblock \showarticletitle{Curvature measures}.
\newblock \bibinfo{journal}{\emph{Trans. Amer. Math. Soc.}}
  \bibinfo{volume}{93} (\bibinfo{year}{1959}), \bibinfo{pages}{418--491}.
\newblock
\showISSN{0002-9947}
\urldef\tempurl%
\url{https://doi.org/10.2307/1993504}
\showDOI{\tempurl}


\bibitem[Horobe\c{t}(2024)]%
        {horobet2024}
\bibfield{author}{\bibinfo{person}{Emil Horobe\c{t}}.}
  \bibinfo{year}{2024}\natexlab{}.
\newblock \showarticletitle{The critical curvature degree of an algebraic
  variety}.
\newblock \bibinfo{journal}{\emph{J. Symbolic Comput.}}  \bibinfo{volume}{121}
  (\bibinfo{year}{2024}), \bibinfo{pages}{Paper No. 102259, 12}.
\newblock
\showISSN{0747-7171,1095-855X}
\urldef\tempurl%
\url{https://doi.org/10.1016/j.jsc.2023.102259}
\showDOI{\tempurl}


\bibitem[Horobe\c{t} and Weinstein(2019)]%
        {horobetweinstein2019}
\bibfield{author}{\bibinfo{person}{Emil Horobe\c{t}} {and}
  \bibinfo{person}{Madeleine Weinstein}.} \bibinfo{year}{2019}\natexlab{}.
\newblock \showarticletitle{Offset hypersurfaces and persistent homology of
  algebraic varieties}.
\newblock \bibinfo{journal}{\emph{Comput. Aided Geom. Design}}
  \bibinfo{volume}{74} (\bibinfo{year}{2019}), \bibinfo{pages}{101767, 14}.
\newblock
\showISSN{0167-8396,1879-2332}
\urldef\tempurl%
\url{https://doi.org/10.1016/j.cagd.2019.101767}
\showDOI{\tempurl}


\bibitem[Jeronimo et~al\mbox{.}(2013)]%
        {jeronimoperruccitsigaridas2013}
\bibfield{author}{\bibinfo{person}{Gabriela Jeronimo}, \bibinfo{person}{Daniel
  Perrucci}, {and} \bibinfo{person}{Elias Tsigaridas}.}
  \bibinfo{year}{2013}\natexlab{}.
\newblock \showarticletitle{On the minimum of a polynomial function on a basic
  closed semialgebraic set and applications}.
\newblock \bibinfo{journal}{\emph{SIAM J. Optim.}} \bibinfo{volume}{23},
  \bibinfo{number}{1} (\bibinfo{year}{2013}), \bibinfo{pages}{241--255}.
\newblock
\showISSN{1052-6234}
\urldef\tempurl%
\url{https://doi.org/10.1137/110857751}
\showDOI{\tempurl}


\bibitem[Niyogi et~al\mbox{.}(2008)]%
        {nswB}
\bibfield{author}{\bibinfo{person}{Partha Niyogi}, \bibinfo{person}{Stephen
  Smale}, {and} \bibinfo{person}{Shmuel Weinberger}.}
  \bibinfo{year}{2008}\natexlab{}.
\newblock \showarticletitle{Finding the homology of submanifolds with high
  confidence from random samples}.
\newblock \bibinfo{journal}{\emph{Discrete Comput. Geom.}}
  \bibinfo{volume}{39}, \bibinfo{number}{1-3} (\bibinfo{year}{2008}),
  \bibinfo{pages}{419--441}.
\newblock
\showISSN{0179-5376,1432-0444}
\urldef\tempurl%
\url{https://doi.org/10.1007/s00454-008-9053-2}
\showDOI{\tempurl}


\bibitem[Niyogi et~al\mbox{.}(2011)]%
        {nswA}
\bibfield{author}{\bibinfo{person}{P. Niyogi}, \bibinfo{person}{S. Smale},
  {and} \bibinfo{person}{S. Weinberger}.} \bibinfo{year}{2011}\natexlab{}.
\newblock \showarticletitle{A topological view of unsupervised learning from
  noisy data}.
\newblock \bibinfo{journal}{\emph{SIAM J. Comput.}} \bibinfo{volume}{40},
  \bibinfo{number}{3} (\bibinfo{year}{2011}), \bibinfo{pages}{646--663}.
\newblock
\showISSN{0097-5397,1095-7111}
\urldef\tempurl%
\url{https://doi.org/10.1137/090762932}
\showDOI{\tempurl}


\bibitem[Renegar(1987)]%
        {renegar1987}
\bibfield{author}{\bibinfo{person}{J. Renegar}.}
  \bibinfo{year}{1987}\natexlab{}.
\newblock \showarticletitle{On the efficiency of {N}ewton's method in
  approximating all zeros of a system of complex polynomials}.
\newblock \bibinfo{journal}{\emph{Math. Oper. Res.}} \bibinfo{volume}{12},
  \bibinfo{number}{1} (\bibinfo{year}{1987}), \bibinfo{pages}{121--148}.
\newblock
\showISSN{0364-765X}
\urldef\tempurl%
\url{https://doi.org/10.1287/moor.12.1.121}
\showDOI{\tempurl}


\bibitem[Rudelson and Vershynin(2015)]%
        {RV-1}
\bibfield{author}{\bibinfo{person}{Mark Rudelson} {and} \bibinfo{person}{Roman
  Vershynin}.} \bibinfo{year}{2015}\natexlab{}.
\newblock \showarticletitle{Small ball probabilities for linear images of
  high-dimensional distributions}.
\newblock \bibinfo{journal}{\emph{Int. Math. Res. Not. IMRN}}
  \bibinfo{volume}{19} (\bibinfo{year}{2015}), \bibinfo{pages}{9594--9617}.
\newblock
\showISSN{1073-7928}
\urldef\tempurl%
\url{https://doi.org/10.1093/imrn/rnu243}
\showDOI{\tempurl}


\bibitem[Tonelli-Cueto and Tsigaridas(2020)]%
        {TCT-cube1conf}
\bibfield{author}{\bibinfo{person}{Josu\'{e} Tonelli-Cueto} {and}
  \bibinfo{person}{Elias Tsigaridas}.} \bibinfo{year}{[2020] \copyright
  2020}\natexlab{}.
\newblock \showarticletitle{Condition numbers for the cube. {I}: {U}nivariate
  polynomials and hypersurfaces}. In
  \bibinfo{booktitle}{\emph{I{SSAC}'20---{P}roceedings of the 45th
  {I}nternational {S}ymposium on {S}ymbolic and {A}lgebraic {C}omputation}}.
  \bibinfo{publisher}{ACM}, \bibinfo{address}{New York},
  \bibinfo{pages}{434--441}.
\newblock


\bibitem[Tonelli-Cueto and Tsigaridas(2023)]%
        {TCT-cube1}
\bibfield{author}{\bibinfo{person}{Josu\'{e} Tonelli-Cueto} {and}
  \bibinfo{person}{Elias Tsigaridas}.} \bibinfo{year}{2023}\natexlab{}.
\newblock \showarticletitle{Condition numbers for the cube. {I}: {U}nivariate
  polynomials and hypersurfaces}.
\newblock \bibinfo{journal}{\emph{J. Symbolic Comput.}}  \bibinfo{volume}{115}
  (\bibinfo{year}{2023}), \bibinfo{pages}{142--173}.
\newblock
\showISSN{0747-7171}
\urldef\tempurl%
\url{https://doi.org/10.1016/j.jsc.2022.08.013}
\showDOI{\tempurl}


\bibitem[Vershynin(2018)]%
        {vershyninbook}
\bibfield{author}{\bibinfo{person}{R. Vershynin}.}
  \bibinfo{year}{2018}\natexlab{}.
\newblock \bibinfo{booktitle}{\emph{High-dimensional probability: An
  introduction with applications in data science}}. \bibinfo{series}{Cambridge
  Series in Statistical and Probabilistic Mathematics},
  Vol.~\bibinfo{volume}{47}.
\newblock \bibinfo{publisher}{Cambridge University Press, Cambridge}.
\newblock
\showISBNx{978-1-108-41519-4}
\urldef\tempurl%
\url{https://doi.org/10.1017/9781108231596}
\showDOI{\tempurl}


\end{thebibliography}
